\documentclass[12pt, a4paper]{article}

\usepackage{latexsym}
\usepackage{booktabs}
\usepackage{caption}
\usepackage{mathrsfs}
\usepackage{amsmath}
\usepackage{amsfonts,amsthm,amssymb,mathrsfs,bbding}
\usepackage{float}
\usepackage{graphicx}
\usepackage{enumerate}
\usepackage{tikz}
\usepackage{caption}
\usepackage{rotating}
\allowdisplaybreaks[4]
\usepackage{tabularx}
\usepackage{cite}
\newtheorem{thm}{Theorem}[section]
\newtheorem{lem}[thm]{Lemma}

\renewcommand\proofname{\bf Proof}
\addtocounter{section}{0}

\theoremstyle{definition}

\newtheorem{const}{Construction}[section]
\newtheorem{prob}{Problem}[section]

\renewcommand\proofname{\bf Proof}

\usepackage{url}

\usepackage[unicode={true},colorlinks,
linkcolor=black,
anchorcolor=blue,
citecolor=green]{hyperref}

\usepackage{authblk}
\usepackage{color}

\definecolor{VeryLightBlue}{rgb}{0.9,0.9,1}
\definecolor{LightBlue}{rgb}{0.8,0.8,1}
\definecolor{MidBlue}{rgb}{0.5,0.5,1}
\definecolor{DarkBlue}{rgb}{0,0,0.6}
\definecolor{Blue}{rgb}{0,0,1}
\definecolor{Gold}{rgb}{1,0.843,0}
\definecolor{LightGreen}{rgb}{0.88,1,0.88}
\definecolor{MidGreen}{rgb}{0.6,1,0.6}
\definecolor{DarkGreen}{rgb}{0,0.6,0}
\definecolor{VeryLightYellow}{rgb}{1,1,0.9}
\definecolor{LightYellow}{rgb}{1,1,0.6}
\definecolor{MidYellow}{rgb}{1,1,0.5}
\definecolor{DarkYellow}{rgb}{1,1,0.2}
\definecolor{DarkPurple}{rgb}{.6,0,1}
\definecolor{Red}{rgb}{1,0,0}
\definecolor{VeryLightRed}{rgb}{1,0.9,0.9}
\definecolor{LightRed}{rgb}{1,0.8,0.8}
\definecolor{MidRed}{rgb}{1,0.55,0.55}

\long\def\delete#1{}

\usepackage[normalem]{ulem}

\begin{document}

\title{Maximum spectral gap of regular graphs with bounded essential edge-connectivity}

\author[1]{Yu Wang\thanks{Corresponding author.}}
\author[2]{Sanming Zhou}

\affil[1]{\small College of Mathematics and System Science, Xinjiang University, Urumqi 830017, P. R. China}
\affil[2]{\small School of Mathematics and Statistics, The University of Melbourne, Parkville, VIC 3010, Australia}
 
\date{}

\openup 0.5\jot 
\maketitle

\renewcommand{\thefootnote}{\empty}
\footnotetext{E-mail addresses: yuwang980212@163.com (Y. Wang), sanming@unimelb.edu.au (S. Zhou)}

\begin{abstract}
An edge-cut of a graph is said to be essential if its removal results in a graph with at least two non-trivial components. The essential edge-connectivity of a graph $G$ is the minimum cardinality among all essential edge-cuts of $G$. The spectral gap of $G$ is the difference between its largest and second largest eigenvalues. In this paper, we prove that for any integers $t$ and $r$ with $6\leq r\leq t\leq 2r-3$, the maximum spectral gap among all connected $r$-regular graphs with essential edge-connectivity at most $t$ is equal to $\frac{1}{2}(r+7-\sqrt{(r+7)^2-8t-32})$ when $t-r$ is odd and $\frac{1}{2}(r+6-\sqrt{(r+6)^2-8t-32})$ when $t-r$ is even. We construct a family of connected $r$-regular graphs achieving these bounds. 



\end{abstract}

\section{Introduction}
\label{sec:int}

In this paper we only consider finite undirected graphs with no loops and no parallel edges only. As usual, we use $V(G)$ and $E(G)$ to denote the vertex and edge sets of a graph $G$, respectively. For subsets $S, T$ of $V(G)$, we use $G[S]$ to denote the subgraph of $G$ induced by $S$, $G - S$ to denote the subgraph of $G$ obtained from $G$ by deleting all vertices in $S$, and $[S, T]$ to denote the set of edges of $G$ with one end-vertex in $S$ and the other end-vertex in $T$. Set $\overline{S} = V(G) \setminus S$ and denote the complement of $G$ by $\overline{G}$. For a subset $X$ of $E(G)$, we use $G - X$ to denote the spanning subgraph of $G$ with edge set $E(G) \setminus X$. We use $\kappa(G)$ and $\kappa'(G)$ to denote the connectivity and edge-connectivity of $G$, respectively. An edge-cut $X$ of $G$ is said to be \emph{essential} if $G - X$ has at least two non-trivial components. The \emph{essential edge-connectivity} of a non-trivial graph $G$, written $\lambda'(G)$, is the minimum cardinality of an essential edge-cut of $G$, whereas the essential edge-connectivity of the trivial graph $K_1$ is defined to be $0$. It follows from this definition that $\kappa'(G)\leq\lambda'(G)$ for any graph $G$. It is well known that $\kappa(G) \leq \kappa'(G) \leq \delta(G)$ for any graph $G$, where $\delta(G)$ is the minimum degree of $G$. However, the inequality $\lambda'(G)\leq \delta(G)$ is not true in general. The essential edge-connectivity of graphs has a wide range of applications, including in the study of spanning trees \cite{X.F.,X.,H-J.}, spanning trails\cite{J.Q.}, connected even factors \cite{J.E}, etc. For more information about the essential edge-connectivity, the reader is referred to \cite{F.L.,N.K}. 

Let $A(G)$ be the adjacency matrix of $G$. The eigenvalues of $A(G)$ are defined as the eigenvalues of $G$. Denote these eigenvalues by $\lambda_1(G) \ge \lambda_2(G) \ge \cdots \ge \lambda_n(G)$, so $\lambda_i(G)$ is the $i$th largest eigenvalue of $G$, where $n=|V(G)|$. The Laplacian matrix of $G$ is defined as $L(G) = D(G) - A(G)$, and its eigenvalues are known as the Laplacian eigenvalues of $G$, where $D(G)$ is the diagonal matrix with diagonal entries the vertex degrees of $G$. Denote these eigenvalues by $\mu_1(G) \le \mu_2(G) \le \cdots \le \mu_n(G)$, so $\mu_i(G)$ is the $i$th smallest Laplacian eigenvalue of $G$. Since $L(G)$ is semi-definite and $L(G) \mathbf{1} = \mathbf{0}$, we have $\mu_1(G) = 0$, where $\mathbf{1}$ and $\mathbf{0}$ are all-one and all-zero vectors, respectively. If $G$ is connected, then $\mu_2(G) > 0$. In 1973, Fiedler \cite{fiedler01} proved that $\mu_2(G) \leq \kappa(G)$ for any non-complete graph $G$, and he called $\mu_2(G)$ the \emph{algebraic connectivity} of $G$. In the case when $G$ is $r$-regular, where $r$ is a non-negative integer, we have $\lambda_1(G) = r$ and $\mu_i(G) = r - \lambda_i(G)$ for all $1 \le i \le n$, and so the algebraic connectivity $\mu_2(G)$ is the \emph{spectral gap} $r - \lambda_2(G)$ of $G$.  

It is widely known that the spectral gap of a graph plays an important role in many applications. For example, it determines how fast a random walk on a graph converges to its stationary distribution. The expansion of a graph $G$ can be measured by its \emph{isoperimetric number} of $G$, which is defined as $h(G) = \min\left\{\frac{|[S, \overline{S}]|}{|S|}: \emptyset \ne S \subset V(G), |S| \le \frac{|V(G)|}{2}\right\}$. It is well known that the isoperimetric number of a connected $r$-regular graph $G$ is determined by its spectral gap owing to the following inequalities (\cite{AM,M}): $\frac{r-\lambda_{2}(G)}{2} \leq h(G) \leq \sqrt{2r(r-\lambda_{2}(G))}$. Thus it is of great importance to study the second largest eigenvalue of connected regular graphs, especially in the context of expanders (see \cite{hoory90}), where, roughly speaking, an expander is a graph with small degree and large isoperimetric number. Because of the importance of the spectral gap, it is natural to study the following extremal problems. 

\begin{prob}
\label{prob}
Determine the maximum spectral gap of graphs in a given family of $r$-regular graphs. 
\end{prob}

\begin{prob}
\label{prob1}
Determine the minimum spectral gap of graphs in a given family of $r$-regular graphs.  
\end{prob}

Of course, Problem \ref{prob} (Problem \ref{prob1}, respectively) is equivalent to the problem of determining the minimum (maximum, respectively) second largest eigenvalue of graphs in a given family of $r$-regular graphs. A related problem is to derive the best lower bound (upper bound, respectively) for the second largest eigenvalue of graphs in the given family.

Problem \ref{prob1} has been studies extensively for various families of regular graphs. For example, Aldous and Fill \cite[p.217]{Aldous02} conjectured that the maximum relaxation time for the random walk on a connected regular graph with $n$ vertices is $(1+o(1)) \frac{3n^2}{2\pi^2}$. As seen in \cite[Conjecture 1.2]{Abdi02} and \cite[Conjecture 1.3]{Abdi03}, this conjecture can be rephrased as follows: The spectral gap of an $r$-regular graph with $n$ vertices is at least $(1+o(1)) \frac{2r\pi^2}{3n^2}$, and this bound is attained for at least one value of $r$. This conjecture has been proved for cubic graphs by Abdi, Ghorbani and Imrich (see \cite[Theorem 2.2]{Abdi03}) and for $r$-regular graphs with diameter $\frac{3n}{r+1}+ O(1)$ by Abdi and Ghorbani (see \cite[Theorem 1.8]{Abdi02}).  

Problem \ref{prob} has been studied for several families of regular graphs. For example, for given integers $t$ and $r$ with $1 \leq t \leq r-1$, Abiad \textit{et al.} \cite{Abiad01} studied the problem of finding the best upper bound for $\lambda_2(G)$ that guarantees $\kappa'(G) \geq t+1$ (or $\kappa(G) \geq t+1$) among all $r$-regular graphs. Equivalently, this asks for the best upper bound on the spectral gap of the graphs in the family of $r$-regular graphs with $\kappa'(G) \le t$ (or $\kappa(G) \leq t$). In 2002, Kirkland \textit{et al.} \cite{kirkland} gave a characterization of graphs with algebraic connectivity achieving $\kappa(G)$. In 2020, O \cite{O10} extended Fiedler's bound $\mu_2(G) \leq \kappa(G)$ to multigraphs. In 2004, Chandran \cite{Chandran01} showed that if $G$ is an $r$-regular graph with order $n$ and $\kappa'(G) \leq r-1$, then $\lambda_2(G) \geq r - 1 - \frac{r}{n-r}$. In 2006, Krivelevich and Sudakov \cite{krivelevich} slightly improved this result as follows: If $G$ is an $r$-regular graph with $\kappa'(G) \leq r-1$, then $\lambda_2(G) > r - 2$. In 2010, Cioab\u{a} \cite{cioaba10} proved that if $G$ is an $r$-regular graph with $\kappa'(G) \leq t$, then $\lambda_2(G) \geq r - \frac{2t}{r+1}$. He proved further that this result can be strengthened as follows when $t$ is $1$ or $2$: If $r\geq 3$ is odd and $G$ is an $r$-regular graph with $\kappa'(G) \leq 1$, then $\lambda_2(G) \geq \pi(r)$, where $\pi(r)$ is the largest root of the equation $x^3 - (r - 3)x^2 - (3r - 2)x - 2 = 0$; if $r\geq 3$ and $G$ is an $r$-regular graph with $\kappa'(G) \leq 2$, then $\lambda_2(G) \geq \frac{r-3+\sqrt{(r+3)^2-16}}{2}$. In 2023, O \textit{et al.} \cite{O90} proved that if $3 \leq t \leq r-1$ and $G$ is an $r$-regular graph with $\kappa'(G) \leq t$, then $\lambda_2(G) \geq \tau(r, t)$, where
\begin{equation}
	\label{tau}
\tau(r, t) = \begin{cases} 
	\frac{r-4+\sqrt{(r+4)^2-8t}}{2}, & \text{when } t \text{ is odd,} \\
	\frac{r-3+\sqrt{(r+3)^2-8t}}{2}, & \text{when } t \text{ is even.}
\end{cases}
\end{equation}
The reader is referred to \cite{deAbreu01,mohar01,rad01,Wu01} for more results on the algebraic connectivity of graphs. 

In this paper we focus on Problem \ref{prob} for a particular family of regular graphs. Inspired by the work in \cite{O90}, we resolve Problem \ref{prob} for the family of $r$-regular graphs $G$ with bounded essential edge-connectivity, where $r \ge 6$. Observe that if $1\leq \lambda'(G) \leq r-1$, then $\lambda_2(G) \ge \tau(r,t)$, where $\tau(r,t)$ is defined in \eqref{tau}. Hence it remains to consider the case when $\lambda'(G)\geq r$. Therefore, in our main result, stated below, we can assume $t \ge r$.

\begin{thm}\label{thm::1.4}
Let $t$ and $r$ be integers with $6 \leq r \leq t\leq 2r-3$, and let 
\begin{equation}
	\label{rho}
\rho(r, t) = \begin{cases} 
\frac{r-7+\sqrt{(r+7)^2-8t-32}}{2}, & \text{when } t-r \text{ is odd,} \\
\frac{r-6+\sqrt{(r+6)^2-8t-32}}{2}, & \text{when } t-r \text{ is even.}
\end{cases}
\end{equation}
Then for any connected $r$-regular graph $G$ with $\lambda'(G) \leq t$, we have $\lambda_2(G) \geq \rho(r, t)$. Moreover, this bound is sharp. 
\end{thm}

Equivalently, Theorem \ref{thm::1.4} says that for $6 \leq r \leq t$ the maximum spectral gap among all connected $r$-regular graphs with essential edge-connectivity at most $t$ is equal to $\frac{r+7-\sqrt{(r+7)^2-8t-32}}{2}$ when $t-r$ is odd and $\frac{r+6-\sqrt{(r+6)^2-8t-32}}{2}$ when $t-r$ is even.

In the next section we will construct a family of connected $r$-regular graphs $G$ satisfying $\lambda'(G) = t$ and $\lambda_2(G) = \rho(r, t)$. The proof of Theorem \ref{thm::1.4} will be given in Section \ref{sec:pf}.

\section{Construction of extremal graphs}
\label{sec:const}

In this section, we construct a family of connected $r$-regular graphs $G_{r,t}$ such that $\lambda'(G_{r,t}) = t$ and $\lambda_2(G_{r,t}) = \rho(r,t)$, for any integers $r$ and $t$ with $6 \leq r \leq t \leq 2r - 3$ such that $r$ is odd when $t$ is odd. Here the parity condition on $r$ and $t$ is necessary because, by the handshaking lemma, if the essential edge-connectivity of a connected $r$-regular graph is odd, then $r$ must be odd. We begin with the following two lemmas on the essential edge-connectivity of regular graphs.

\begin{lem}\label{pro::2.1}
If $G$ is a connected $r$-regular graph, where $r \ge 1$, then $\lambda'(G)\leq 2r-2$. Moreover, this bound is achieved by $K_{r+1}$.
\end{lem}

\begin{proof}
Suppose to the contrary that $\lambda'(G) \ge 2r-1$. Let $X$ be a minimum essential edge-cut of $G$, and let $G_1$ and $G_2$ be two non-trivial components of $G-X$. Since $G$ is connected and $r$-regular, we have $|V(G_1)|\ge 3$ and $|V(G_2)|\ge 3$.

Choose two adjacent vertices $x, y \in V(G_1)$. Let $Y$ be the set of edges of $G$ between $\{x, y\}$ and $V(G)\setminus\{x, y\}$. Then $Y$ is an essential edge-cut of $G$ because $\{x,y\}$ induces a connected subgraph and $V(G)\setminus\{x,y\}$ contains the non-empty set $V(G_2)$. Since $G$ is $r$-regular, we have $|Y| = 2r-2 < 2r-1 \le \lambda'(G)$, which contradicts the minimality of $\lambda'(G)$. Therefore, $\lambda'(G) \le 2r-2$.
This bound is tight as $\lambda'(K_{r+1}) = 2r-2$.
\end{proof}

In our study of the maximum spectral gap, we only consider connected $r$-regular graphs $G$ with $\lambda'(G)\leq 2r-3$.

\begin{lem}\label{pro::2.2}
Let $G$ be a connected $r$-regular graph with $r \leq \lambda'(G) \leq 2r-3$, and let $S\subseteq V(G)$ be such that $|[S, \overline{S}]| = \lambda'(G)$. Then each of $S$ and $\overline{S}$ has at least $r$ vertices. Furthermore, if $\lambda'(G)-r$ is odd, then each of $S$ and $\overline{S}$ has at least $r + 1$ vertices.
\end{lem}

\begin{proof}
Denote $l = \lambda'(G)$. If $|S| \leq r-1$, then $2r - 3 \geq l \geq |S| (r + 1 - |S|) \geq 2r-2$, which is a contradiction. Hence $|S|\geq r$. Similarly, $|\overline{S}|\geq r$. 

Now we assume that $l-r$ is odd. If $|S| = r$, then $r |S| = r^2 = 2|E(G[S])| + l$,
that is, $r(r-1) = 2|E(G[S])| + (l-r)$, which is a contradiction as the two sides have different parity. Hence $|S|\geq r+1$. Similarly, $|\overline{S}|\geq r+1$. 
\end{proof}

As usual, for two graphs $G$ and $H$ and a positive integer $k$, $G\cup H$ denotes the \emph{union} of $G$ and $H$, and $kG$ denotes the union of $k$ pairwise vertex-disjoint copies of $G$. If $G$ and $H$ are vertex-disjoint, $G\vee H$ denotes the \emph{join} of $G$ and $H$, which is obtained from $G\cup H$ by adding all possible edges between the vertices of $G$ and the vertices of $H$. Recall that $\overline{G}$ denotes the complement of $G$.

\begin{const}
\label{const1}
Let $r$ and $t$ be positive integers with $r \leq t \leq 2r - 3$ such that $r$ is odd when $t$ is odd. Then $t$ is even when $t-r$ is odd. Let
\begin{equation}
\label{Hrt}
H_{r,t} =
\begin{cases}
\overline{C_{t-r-1}} \vee \overline{\frac{2r-t+2}{2}K_{2}}, & \text{if } t-r \text{ is odd}, \\
\overline{\frac{t-r}{2}K_2} \vee K_{2r-t},   & \text{if } t-r \text{ is even}.
\end{cases}
\end{equation}

If $t-r$ is odd, then we define $G_{r,t}$ to be a graph obtained from the union of two vertex-disjoint copies of $H_{r,t}$ by (i) adding $2(t-r-1)$ edges between the two copies of $\overline{C_{t-r-1}}$ such that every vertex in each copy of $\overline{C_{t-r-1}}$ has exactly two neighbors in the other copy of $\overline{C_{t-r-1}}$, and (ii) adding a perfect matching between the two copies of $\overline{\frac{2r-t+2}{2}K_{2}}$.

If $t-r$ is even, then we define $G_{r,t}$ to be a graph obtained from the union of two vertex-disjoint copies of $H_{r,t}$ by (i) adding $2(t-r)$ edges between the two copies of $\overline{\frac{t-r}{2}K_2}$ such that every vertex in each copy of $\overline{\frac{t-r}{2}K_2}$ has exactly two neighbors in the other copy of  $\overline{\frac{t-r}{2}K_2}$, and (ii) adding a perfect matching between the two copies of $K_{2r-t}$.  

It can be easily verified that $G_{r,t}$ is a connected $r$-regular graph regardless of the parity of $t-r$. The order of $G_{r,t}$ is $2(r+1)$ when $t-r$ is odd and $2r$ when $t-r$ is even. 
\end{const}

It is worth noting that the graph $G_{r,t}$ constructed above is not unique. We use ${\mathcal G}_{r,t}$ to denote the family of graphs $G_{r,t}$ defined in Construction \ref{const1}. The following lemma gives the essential edge-connectivity of each $G_{r,t}$. 

\begin{lem}\label{pro::2.3}
Let $r$ and $t$ be integers with $6\leq r \leq t \leq 2r - 3$ such that $r$ is odd when $t$ is odd. Then $\lambda'(G_{r,t}) = t$ for any $G_{r,t} \in {\mathcal G}_{r,t}$.
\end{lem}

\begin{proof}
We only prove the result for odd $t-r$. The proof for even $t-r$ is similar and hence is omitted.

Assume that $t-r$ is odd. Set $a = t-r-1$ and $b = \frac{2r-t+2}{2}$. Then $a$ is even and $b \geq 3$ as $t \leq 2r-3$. Since $t-r$ is odd, the graph $H_{r,t}$ defined in \eqref{Hrt} is $H_{r,t} = \overline{C_{a}} \vee \overline{bK_{2}}$. Take two vertex-disjoint copies of $H_{r,t}$. Call them $H_{r,t}^1$ and $H_{r,t}^2$. Denote by $A_i$ and $B_i$ the vertex sets of $C_{a}$ and $\overline{bK_{2}}$ in $H_{r,t}^i$, respectively, for $i=1,2$. Recall from Construction \ref{const1} that $G_{r,t}$ is obtained from $H_{r,t}^1 \cup H_{r,t}^2$ by adding the $2a$ edges of a $2$-regular bipartite graph with bipartition $\{A_1, A_2\}$ and a perfect matching between $B_1$ and $B_2$. Note that $V(H_{r,t}^i) = A_i \cup B_i$ for $i=1,2$ and $V(G_{r,t}) = \cup_{i=1}^{2} (A_i \cup B_i)$. Note also that there are exactly $2a+2b=t$ edges of $G_{r,t}$ between $H_{r,t}^1$ and $H_{r,t}^2$. These $t$ edges form an essential edge-cut of $G_{r,t}$ as their removal disconnects $G_{r,t}$ into two nontrivial components. Thus, $\lambda'(G_{r,t}) \leq t$.

It remains to prove that no essential edge-cut of $G_{r,t}$ with size smaller than $t$ exists. Suppose for a contradiction that there exists an essential edge-cut $X$ of $G_{r,t}$ such that $|X| < t$. Then $G_{r,t}-X$ has at least two nontrivial components.
Let $S$ be the vertex set of one of these components and write $S_i = S \cap V(H_{r,t}^i)$ for $i=1,2$. Set $x_i = |S_i \cap A_i|$ and $y_i = |S_i \cap B_i|$. Then $0 \leq x_i \leq a$ and $0 \leq y_i \leq 2b$ for $i=1,2$.

Inside $H_{r,t}^i$ we have
\[
|[S_i, V(H_{r,t}^i) \setminus S_i]| \geq 2b x_i + a y_i - 2x_i y_i .
\]
Regarding the edges between $H_{r,t}^1$ and $H_{r,t}^2$, for $i=1,2$, each vertex in $A_i$ contributes exactly two edges incident to some vertices in $A_{3-i}$, and each vertex in $B_i$ contributes exactly one edge incident to a vertex in $B_{3-i}$. Therefore, the number of edges leaving $S$ is at least $2|x_1 - x_2| + |y_1 - y_2|$. Hence
\begin{align}
|[S,V(G_{r,t})\setminus S]| \geq \sum_{i=1}^{2} \bigl(2b x_i + a y_i - 2x_i y_i\bigr) 
+ 2|x_1 - x_2| + |y_1 - y_2|. \label{equ::7}
\end{align}
Since the right-hand side of (\ref{equ::7}) is bilinear in each variable on regions where the signs of $x_1-x_2$ and $y_1-y_2$ are fixed, its minimum value over the continuous domain $0 \leq x_i \leq a$, $0 \leq y_i \leq 2b$ occurs at a corner of the rectangle $[0,a]\times[0,2b]$. Evaluating at all corners yields that the minimum nonzero value of the right-hand side of (\ref{equ::7}) equals $2a+2b = t$, and this value is attained only when $(x_1,y_1) = (a,2b)$ and $(x_2,y_2) = (0,0)$, or $(x_2,y_2) = (a,2b)$ and $(x_1,y_1) = (0,0)$. These combinations correspond to $S$ being exactly $V(H_{r,t}^1)$ or $V(H_{r,t}^2)$. For any other choice of $(x_i,y_i)$, the right-hand side of (\ref{equ::7}) exceeds $t$. In particular, for any essential edge-cut the quantity in (\ref{equ::7}) cannot take any positive value strictly smaller than $t$. Consequently, we have $|[S,V(G_{r,t}) \setminus S]| \geq t$, and the equality holds only when $S$ is $V(H_{r,t}^1)$ or $V(H_{r,t}^2)$. However, if $S$ is $V(H_{r,t}^1)$ or $V(H_{r,t}^2)$, then $|[S,V(G_{r,t}) \setminus S]| = t$, which contradicts the assumption that $|X| < t$. This contradiction shows that no essential edge-cut of $G_{r,t}$ with size smaller than $t$ exists, and therefore $\lambda'(G_{r,t}) = t$.
\end{proof}

Next we determine the second largest eigenvalue of $G_{r,t}$. Before doing so, let us recall an important tool, namely the Quotient Interlacing Theorem.
Let $P = \{V_1, \ldots, V_s\}$ be a partition of the vertex set of a graph $G$ into $s$ non-empty subsets. The quotient matrix $Q$ corresponding to $P$ is the $s \times s$ matrix whose $(i,j)$-entry $Q_{i,j}\ (1 \leq i, j \leq s)$ is the average number of neighbors in $V_j$ of the vertices in $V_i$. In other words, $Q_{i,j} = \frac{|[V_i, V_j]|}{|V_i|}$ if $i \neq j$, and $Q_{i,i} = \frac{2|E(G[V_i])|}{|V_i|}$. If, for any $1 \leq i, j \leq s$, each vertex in $V_i$ has exactly $Q_{i,j}$ neighbors in $V_j$, then $P$ is called an \emph{equitable partition} of $G$.

\begin{thm}[{\cite[Quotient Interlacing Theorem]{brouwer11}}]
\label{thm::2.1}
Let $G$ be a graph and $Q$ the quotient matrix corresponding to a partition $P$ of $V(G)$. Then the eigenvalues of $Q$ interlace the eigenvalues of $G$.
\end{thm}

In the special case when $P$ is an equitable partition, all eigenvalues of $Q$ are eigenvalues of $G$ and the spectral radius of $Q$ equals the spectral radius of $G$ (see \cite{brouwer11,godsil01} for details).

\begin{lem}
\label{thm::2.2}
Let $r$ and $t$ be integers with $6\leq r \leq t \leq 2r - 3$ such that $r$ is odd when $t$ is odd. Then for any $G_{r,t} \in {\mathcal G}_{r,t}$ we have
\[
\lambda_2(G_{r,t}) = 
\begin{cases} 
\frac{r-7+\sqrt{(r+7)^2-8t-32}}{2}, & \text{if } t-r \text{ is odd,} \\
 \frac{r-6+\sqrt{(r+6)^2-8t-32}}{2}, & \text{if } t-r \text{ is even.}
\end{cases}
\]
\end{lem}

\begin{proof}
We only prove the result for odd $t-r$. The proof for even $t-r$ is similar and hence is omitted.

Assume that $t-r$ is odd. Denote by $H_1 \vee H_2$ and $H_3 \vee H_4$ the two vertex-disjoint copies of $H_{r,t}$ in $G_{r,t}$, where $H_1$ and $H_3$ are copies of $\overline{C_{t-r-1}}$, and $H_2$ and $H_4$ are copies of $\overline{\frac{2r-t+2}{2}K_{2}}$. By the definition of $G_{r,t}$, $\{V(H_1), V(H_2), V(H_3), V(H_4)\}$ is an equitable partition of $G_{r,t}$ with quotient 
matrix 
\[
Q = \begin{pmatrix}
t - r - 4  & 2r - t + 2 & 2 & 0 \\
t - r - 1  & 2r - t     & 0 & 1 \\
2 & 0  & t - r - 4  & 2r - t + 2 \\
0 & 1  & t - r - 1  & 2r - t
\end{pmatrix}.
\]
The eigenvalues of $Q$ are $r$, $-1$ and $\frac{r-7\pm\sqrt{(r+7)^2-8t-32}}{2}$. These are eigenvalues of $G_{r,t}$ as $Q$ is the quotient matrix of an equitable partition. The lifted eigenvectors corresponding to these four eigenvalues form a basis for a subspace of $\mathbb{R}^{2r+2}$. The remaining eigenvectors can be chosen orthogonal to this subspace and hence orthogonal to the characteristic vectors of the four parts. Since these eigenvectors are orthogonal to the characteristic vectors of each part, the sums of the coordinates in every part are zero; hence the edges inside each part act as zero on this subspace, so deleting these edges does not change the eigen-equation. Due to the specific structure of $G_{r,t}$, these remaining eigenvectors are also eigenvectors of the graph obtained by deleting all edges among the four parts, which is the induced subgraphs $H_1, H_2, H_3, H_4$ of $G_{r,t}$.

We now analyze the spectra of $H_1, H_2, H_3$ and $H_4$. Since $H_1$ and $H_3$ are isomorphic to $\overline{C_{t-r-1}}$, their eigenvalues lie in the interval $[-3, 1]$, with the largest eigenvalue being $t-r-4$, and the nontrivial eigenvalues (those orthogonal to the all-one vector) are at most $1$. Since $H_2$ and $H_4$ are isomorphic to $\overline{\frac{2r-t+2}{2}K_2}$, their spectra are $\left\{(2r-t)^{(1)}, 0^{\left(\frac{2r-t+2}{2}\right)}, (-2)^{\left(\frac{2r-t}{2}\right)}\right\}$, with $0$ being the largest nontrivial eigenvalue.

The nontrivial eigenvalues of $H_1$ and $H_3$ are at most $1$. For a nontrivial eigenvalue $\lambda$ of $H_1$ with eigenvector $\mathbf{x}$, the vector $(a\mathbf{x}, \mathbf{0}, b\mathbf{x}, \mathbf{0})^T$ yields eigenvalues $\lambda \pm \gamma$ with $|\gamma| \le 2$, hence $\lambda \pm \gamma \leq3$. The nontrivial eigenvalues of $H_2$ and $H_4$ are at most $0$. Using the perfect matching between $H_2$ and $H_4$, the vectors $(\mathbf{0}, \mathbf{y}, \mathbf{0}, \pm\mathbf{y})^T$ give eigenvalues $\eta \pm 1 \le 1$.

All eigenvalues of $G_{r,t}$ arise from the above constructions or are among the four eigenvalues of $Q$. From what we proved above, the eigenvalues from the subgraph constructions are bounded above by $\max\{3, 1\} = 3$. On the other hand, the largest eigenvalue of $Q$ is $r$ (which is also the largest eigenvalue of $G$), and the second largest eigenvalue of $Q$ is $\frac{r-7+\sqrt{(r+7)^2-8t-32}}{2}$, which is greater than $3$ as $6\leq r \leq t \leq 2r-3$ by our assumption. Therefore, the second largest eigenvalue of $G_{r,t}$ is $\lambda_2(G_{r,t}) = \frac{r-7+\sqrt{(r+7)^2-8t-32}}{2}$. This completes the proof.
\end{proof}

\section{Proof of Theorem \ref{thm::1.4}}
\label{sec:pf}

\renewcommand\proofname{\bf Proof of Theorem \ref{thm::1.4}}
\begin{proof}
Let $t$ and $r$ be integers with $6 \leq r \leq t$. Let $G$ be a connected $r$-regular graph with $\lambda'(G) \leq t$. We aim to prove $\lambda_2(G)\geq \rho(r,t)$, where $\rho(r,t)$ is as defined in \eqref{rho}. 

Since $\lambda'(G) \leq t$, there exists a vertex subset $S \subseteq V(G)$ such that $l := |[S, \overline{S}]| \leq t$. Set $s = |S|$ and let $s' = |\overline{S}|$. Then the quotient matrix of the partition $\{S, \overline{S}\}$ of $V(G)$ is
\begin{equation}
	\label{eq:Q}
Q = \left( \begin{array}{cc}
r - \frac{l}{s} & \frac{l}{s} \\[6pt]
\frac{l}{s'} & r - \frac{l}{s'}
\end{array} \right).
\end{equation}
The eigenvalues of $Q$ are $r$ and $r - \frac{l}{s} - \frac{l}{s'}$. So, by Theorem \ref{thm::2.1},
\begin{align}
\lambda_2(G) \geq \lambda_2(Q) = r - \frac{l}{s} - \frac{l}{s'}.\label{equ::1}
\end{align}

If $l \leq t - 1$, then by (\ref{equ::1}) and Lemma \ref{pro::2.2}, for even $l-r$ we have
\[
\lambda_2(G) \geq r - \frac{2l}{r} \geq r - \frac{2(t - 1)}{r} >\frac{r-6+\sqrt{(r+6)^2-8t-32}}{2},
\]
and for odd $l-r$ we have
\[
\lambda_2(G) \geq r - \frac{2l}{r+1} \geq r - \frac{2(t - 1)}{r+1} >\frac{r-7+\sqrt{(r+7)^2-8t-32}}{2}.
\]
Thus, $\lambda_2(G)>\rho(r,t)$ if $l \leq t - 1$.

In the remaining proof we assume that $l=t$, that is, $|[S, \overline{S}]|=t$. Since $l = t$, by (\ref{equ::1}) we have
\begin{align}
\lambda_2(G) \geq \lambda_2(Q) = r - \frac{t}{s} - \frac{t}{s'}.\label{equ::A}
\end{align}
 
{\flushleft\bf Claim 1.} Assume that $t-r$ is odd. If $s=s'=r+1$, then $\lambda_2(G)\geq r-\frac{2t}{r+1}$. If $s\geq r+2$ or $s'\geq r+2$, then $\lambda_2(G)>\rho(r,t)$.

{\flushleft\bf Proof.}
Assume that $t-r$ is odd. Since $|[S,\overline S]|=t$, Lemma~\ref{pro::2.2} implies that $s,s'\ge r+1$. Obviously, if $s=s'=r+1$, then by \eqref{equ::A},  $\lambda_2(G) \geq r - \frac{2t}{r+1}$. We now prove that if either $s$ or $s'$ is no less than $r+2$, then $\lambda_2(G)>\rho(r,t)$. 
Suppose that $s \ge r+2$. Since $s'\geq r+1$, by \eqref{equ::A} we have
\begin{align}
\lambda_2(G) \geq \lambda_2(Q) \geq r - \frac{t}{r+2} - \frac{t}{r+1}.\label{equ::3}
\end{align}
Note that $6 \le r \le t\leq 2r-3$. Define
\begin{align*}
f(r,t) &=r-\frac{t}{r+2}-\frac{t}{r+1}-\frac{r-7+\sqrt{(r+7)^2-8t-32}}{2}\\
&=\frac{r+7}{2}-t\Bigl(\frac{1}{r+2}+\frac{1}{r+1}\Bigr)-\frac{\sqrt{(r+7)^2-8t-32}}{2}\\
&=\frac{C}{2}-tD-\frac{\sqrt{C^2-8t-32}}{2},
\end{align*}
where we set
\[
C=r+7,\qquad D=\frac{1}{r+2}+\frac{1}{r+1}.
\]
We claim that $f(r,t)>0$ for all real numbers $r,t$. In fact, since $6\le r\le t\le 2r-3$ by Lemma \ref{pro::2.1}, one can verify that $C-2tD>0$. Hence
$f(r,t)>0\;\Leftrightarrow\; C-2tD > \sqrt{C^2-8t-32}
\;\Leftrightarrow\; (C-2tD)^2 > C^2-8t-32 \;\Leftrightarrow\; 4D^2t^2-4CDt+8t+32 > 0\;\Leftrightarrow\; D^2t^2+(2-CD)t+8 > 0.$
Consider the quadratic polynomial $p(x)=D^2x^2+(2-CD)x+8$. It has leading coefficient $D^2>0$ and discriminant
\[
\Delta=(2-CD)^2-32D^2=\frac{(-11r-17)^2-32(2r+3)^2}{(r+1)^2(r+2)^2}
=\frac{-7r^2-10r+1}{(r+1)^2(r+2)^2} < 0
\]
as $r\ge 6$. Hence $p(x) > 0$ for all real number $x$. This together with the arguments above implies that $f(r,t) > 0$ for all $t$ and $r$ with $6 \leq r \leq t \leq 2r - 3$. Combining this with (\ref{equ::3}), we obtain that
\[
\lambda_2(G) \geq \lambda_2(Q) \geq r - \frac{t}{r+2}- \frac{t}{r+1} > \frac{r-7+\sqrt{(r+7)^2-8t-32}}{2}.
\]
Thus, $\lambda_2(G)>\rho(r,t)$ if $s \geq r+2$. Similarly, if $s' \geq r+2$, then $\lambda_2(G)>\rho(r,t)$. This completes the proof of Claim 1.

{\flushleft\bf Claim 2.} Assume that $t-r$ is even. If $s=s'=r$, then $\lambda_2(Q)\geq r-\frac{2t}{r}$. If $s\geq r+1$ or $s'\geq r+1$, then $\lambda_2(G)>\rho(r,t)$.

{\flushleft\bf Proof.}
Assume that $t-r$ is even. Since $|[S,\overline S]|=t$,
Lemma~\ref{pro::2.2} implies that $s,s'\ge r$. Clearly, if $s=s'=r$, then by \eqref{equ::A}, $\lambda_2(Q)\geq r-\frac{2t}{r}$. We now prove that if $s\geq r+1$ or $s'\geq r+1$, then $\lambda_2(G)>\rho(r,t)$. Suppose that $s\geq r+1$. Since $s' \geq r$, by \eqref{equ::A} we have
\begin{align}
\lambda_2(G) \geq \lambda_2(Q) \geq r - \frac{t}{r+1} - \frac{t}{r}.\label{equ::2}
\end{align}
Note that $6 \le r \le t$. Define
\begin{align*}
f(r,t) & = r - \frac{t}{r+1} - \frac{t}{r} - \frac{r-6 + \sqrt{(r+6)^2 - 8t - 32}}{2} \\
 & = \frac{r+6}{2} - t\left(\frac{1}{r+1} + \frac{1}{r}\right) - \frac{\sqrt{(r+6)^2 - 8t - 32}}{2} \\
 & = \frac{C}{2} - tD - \frac{\sqrt{C^2 - 8t - 32}}{2},
\end{align*}
where we set
$$
C = r + 6\ \text{ and }\ D = \dfrac{1}{r+1} + \dfrac{1}{r}.
$$ 
We claim that $f(r,t)>0$ for all $r$ and $t$ with $6 \le r \le t$. In fact, since $6 \le r \le t$, we have $C - 2tD > 0$. Similarly to the proof of Claim 1, we see that $f(r,t) > 0 \Leftrightarrow D^2 t^2 + (2 - CD)t + 8 > 0$. The quadratic polynomial $p(x) = D^2 x^2 + (2 - CD)x + 8$ has leading coefficient $D^2 > 0$ and discriminant
$$
\Delta = (2 - CD)^2 - 32D^2=\frac{(11r+6)^2 - 32(2r+1)^2}{r^2(r+1)^2} = \frac{-7r^2 + 4r + 4}{r^2(r+1)^2} < 0
$$ 
as $r \ge 6$. Hence $p(x) > 0$ for all real number $x$. This together with the arguments above implies that $f(r,t) > 0$ for all $t$ and $r$ with $6 \leq r \leq t \leq 2r - 3$. Combining this with (\ref{equ::2}), we obtain that
\[
\lambda_2(G) \geq \lambda_2(Q) \geq r - \frac{t}{r+1}- \frac{t}{r} > \frac{r-6+\sqrt{(r+6)^2-8t-32}}{2}.
\]
Thus, $\lambda_2(G)>\rho(r,t)$ if $s \geq r+1$. Similarly, if $s' \geq r+1$, then $\lambda_2(G)>\rho(r,t)$. This completes the proof of Claim 2.

\textbf{Case 1.} $t = s$.

In this case, we have $t=r$ or $t=r+1$ by Claims 1 and 2. If $t=r$, then $H_{r,t}=K_r$ in Construction \ref{const1}, so $G$ is formed by taking two copies of $K_r$ and adding a perfect matching between them. If $t=r+1$, then $H_{r,t}=\overline{\frac{r+1}{2}K_2}$ in Construction \ref{const1}, so $G$ is obtained by taking two copies of $\overline{\frac{r+1}{2}K_2}$ and adding a perfect matching between them. Thus, when $t = s$, we have $G\cong G_{r,t}$ and $\lambda_2(G)=\rho(r,t)$ by Lemma \ref{thm::2.2}.

\textbf{Case 2.} $t \geq s+1$.

Since $s=s'$ by Claims 1 and 2, we may write 
$$
S = \{u_1, u_2, \ldots, u_s\},\; \overline{S} = \{v_1, v_2, \ldots, v_s\}.
$$ 
Define a subset $W$ of $V(G)$ as follows: Whenever a vertex $u_i \in S$ and a vertex $v_j \in \overline{S}$ are adjacent in $G$ such that $v_j$ is the only neighbor of $u_i$ in $\overline{S}$ and $u_i$ is the only neighbor of $v_j$ in $S$, we put both $u_i$ and $v_j$ into $W$. Set 
$$
U_1 = S \setminus W,\; U_2=S\cap W
$$ 
and 
$$
U_3 = \overline{S} \setminus W,\; U_4=\overline{S}\cap W.
$$
Of course, $\{U_1, U_2, U_3, U_4\}$ is a partition of $V(G)$ (see Fig. \ref{fig1}). Set $|U_i| = n_i$ for $1\leq i\leq 4$. Then $n_2=n_4=s-n_1$ by the definition of $W$. Since $s=s'$, it follows that $n_3=n_1$. Note that $n_1 \leq t-r$. Moreover, by Lemma \ref{pro::2.2}, if $t-r$ is odd, then $n_1 \leq t-r-1$. 

\begin{figure}[http]
\centering
\begin{tikzpicture}[x=1.00mm, y=1.00mm, inner xsep=0pt, inner ysep=0pt, outer xsep=0pt, outer ysep=0pt]
\path[line width=0mm] (31.93,48.36) rectangle +(114.66,51.62);
\definecolor{L}{rgb}{0,0,0}
\definecolor{F}{rgb}{0,0,0}
\path[line width=0.15mm, draw=L, fill=F] (43.07,97.48) circle (0.50mm);
\path[line width=0.15mm, draw=L, fill=F] (43.07,81.94) circle (0.50mm);
\path[line width=0.15mm, draw=L, fill=F] (85.81,97.48) circle (0.50mm);
\path[line width=0.15mm, draw=L, fill=F] (85.81,81.94) circle (0.50mm);
\path[line width=0.15mm, draw=L] (43.20,97.37) -- (86.32,97.37);
\path[line width=0.15mm, draw=L] (43.00,81.83) -- (85.93,81.83);
\path[line width=0.15mm, draw=L] (43.20,97.37) -- (43.20,82.41);
\path[line width=0.15mm, draw=L] (85.93,97.56) -- (85.93,82.02);
\path[line width=0.15mm, draw=L, fill=F] (101.35,97.48) circle (0.50mm);
\path[line width=0.15mm, draw=L, fill=F] (101.35,81.94) circle (0.50mm);
\path[line width=0.15mm, draw=L, fill=F] (144.08,97.48) circle (0.50mm);
\path[line width=0.15mm, draw=L, fill=F] (144.08,81.94) circle (0.50mm);
\path[line width=0.15mm, draw=L] (101.47,97.37) -- (144.59,97.37);
\path[line width=0.15mm, draw=L] (101.28,81.83) -- (144.20,81.83);
\path[line width=0.15mm, draw=L] (101.47,97.37) -- (101.47,82.41);
\path[line width=0.15mm, draw=L] (144.20,97.56) -- (144.20,82.02);
\path[line width=0.15mm, draw=L, fill=F] (43.07,66.40) circle (0.50mm);
\path[line width=0.15mm, draw=L, fill=F] (43.07,50.86) circle (0.50mm);
\path[line width=0.15mm, draw=L, fill=F] (85.81,66.40) circle (0.50mm);
\path[line width=0.15mm, draw=L, fill=F] (85.81,50.86) circle (0.50mm);
\path[line width=0.15mm, draw=L] (43.20,66.29) -- (86.32,66.29);
\path[line width=0.15mm, draw=L] (43.00,50.75) -- (85.93,50.75);
\path[line width=0.15mm, draw=L] (43.20,66.29) -- (43.20,51.33);
\path[line width=0.15mm, draw=L] (85.93,66.48) -- (85.93,50.94);
\path[line width=0.15mm, draw=L, fill=F] (101.35,66.40) circle (0.50mm);
\path[line width=0.15mm, draw=L, fill=F] (101.35,50.86) circle (0.50mm);
\path[line width=0.15mm, draw=L, fill=F] (144.08,66.40) circle (0.50mm);
\path[line width=0.15mm, draw=L, fill=F] (144.08,50.86) circle (0.50mm);
\path[line width=0.15mm, draw=L] (101.47,66.29) -- (144.59,66.29);
\path[line width=0.15mm, draw=L] (101.28,50.75) -- (144.20,50.75);
\path[line width=0.15mm, draw=L] (101.47,66.29) -- (101.47,51.33);
\path[line width=0.15mm, draw=L] (144.20,66.48) -- (144.20,50.94);
\path[line width=0.15mm, draw=L, fill=F] (45.92,86.84) circle (0.50mm);
\path[line width=0.15mm, draw=L, fill=F] (49.69,88.24) circle (0.50mm);
\path[line width=0.15mm, draw=L, fill=F] (53.68,86.84) circle (0.50mm);
\path[line width=0.15mm, draw=L, fill=F] (57.64,87.60) circle (0.50mm);
\path[line width=0.15mm, draw=L, fill=F] (61.45,86.84) circle (0.50mm);
\path[line width=0.15mm, draw=L, fill=F] (80.88,86.84) circle (0.50mm);
\path[line width=0.15mm, draw=L, fill=F] (104.11,86.84) circle (0.50mm);
\path[line width=0.15mm, draw=L, fill=F] (109.18,87.92) circle (0.50mm);
\path[line width=0.15mm, draw=L, fill=F] (113.93,86.09) circle (0.50mm);
\path[line width=0.15mm, draw=L, fill=F] (119.65,86.84) circle (0.50mm);
\path[line width=0.15mm, draw=L, fill=F] (139.07,86.84) circle (0.50mm);
\path[line width=0.15mm, draw=L, fill=F] (45.92,60.98) circle (0.50mm);
\path[line width=0.15mm, draw=L, fill=F] (49.91,60.01) circle (0.50mm);
\path[line width=0.15mm, draw=L, fill=F] (54.12,60.01) circle (0.50mm);
\path[line width=0.15mm, draw=L, fill=F] (57.25,60.98) circle (0.50mm);
\path[line width=0.15mm, draw=L, fill=F] (61.45,60.98) circle (0.50mm);
\path[line width=0.15mm, draw=L, fill=F] (80.88,60.98) circle (0.50mm);
\path[line width=0.09mm, draw=L, fill=F] (65.16,73.91) circle (0.30mm);
\path[line width=0.09mm, draw=L, fill=F] (67.32,73.91) circle (0.30mm);
\path[line width=0.09mm, draw=L, fill=F] (69.47,73.91) circle (0.30mm);
\path[line width=0.15mm, draw=L, fill=F] (104.11,60.98) circle (0.50mm);
\path[line width=0.15mm, draw=L, fill=F] (109.18,60.12) circle (0.50mm);
\path[line width=0.15mm, draw=L, fill=F] (114.11,59.69) circle (0.50mm);
\path[line width=0.15mm, draw=L, fill=F] (119.65,60.98) circle (0.50mm);
\path[line width=0.15mm, draw=L, fill=F] (139.07,60.98) circle (0.50mm);
\path[line width=0.09mm, draw=L, fill=F] (126.72,73.91) circle (0.30mm);
\path[line width=0.09mm, draw=L, fill=F] (128.87,73.91) circle (0.30mm);
\path[line width=0.09mm, draw=L, fill=F] (131.03,73.91) circle (0.30mm);
\path[line width=0.15mm, draw=L] (104.19,86.74) -- (104.19,60.98);
\path[line width=0.15mm, draw=L] (109.15,87.92) -- (109.15,60.23);
\path[line width=0.15mm, draw=L] (113.89,85.98) -- (113.89,60.23);
\path[line width=0.15mm, draw=L] (119.71,86.74) -- (119.71,61.20);
\path[line width=0.15mm, draw=L] (139.21,86.63) -- (139.21,60.98);
\draw(53.11,92.66) node[anchor=base west]{\fontsize{8.54}{10.24}\selectfont $U_1$};
\draw(128.55,92.66) node[anchor=base west]{\fontsize{8.54}{10.24}\selectfont $U_2$};
\draw(128.87,54.62) node[anchor=base west]{\fontsize{8.54}{10.24}\selectfont $U_4$};
\draw(53.44,54.41) node[anchor=base west]{\fontsize{8.54}{10.24}\selectfont $U_3$};
\path[line width=0.15mm, draw=L] (45.89,86.74) -- (45.89,61.20);
\path[line width=0.15mm, draw=L] (49.88,88.14) -- (45.89,61.09);
\path[line width=0.15mm, draw=L] (49.77,88.14) -- (54.08,60.01);
\path[line width=0.15mm, draw=L] (49.88,59.90) -- (53.65,86.84);
\path[line width=0.15mm, draw=L] (57.21,60.87) -- (57.64,87.71);
\path[line width=0.15mm, draw=L] (61.20,86.95) -- (49.88,59.69);
\path[line width=0.15mm, draw=L] (61.30,86.74) -- (61.20,61.09);
\path[line width=0.15mm, draw=L] (81.02,86.95) -- (81.02,61.09);
\path[line width=0.15mm, draw=L] (81.02,86.63) -- (61.20,60.76);
\draw(64.75,92.55) node[anchor=base west]{\fontsize{8.54}{10.24}\selectfont $n_1$};
\draw(117.02,92.45) node[anchor=base west]{\fontsize{8.54}{10.24}\selectfont $n_2$};
\draw(116.91,54.51) node[anchor=base west]{\fontsize{8.54}{10.24}\selectfont $n_4$};
\draw(64.97,54.30) node[anchor=base west]{\fontsize{8.54}{10.24}\selectfont $n_3$};
\draw(33.93,88.14) node[anchor=base west]{\fontsize{8.54}{10.24}\selectfont $S$};
\draw(33.93,57.96) node[anchor=base west]{\fontsize{8.54}{10.24}\selectfont $\overline{S}$};
\path[line width=0.24mm, draw=L] (88.67,89.86) -- (98.05,89.86);
\path[line width=0.24mm, draw=L, fill=F] (88.67,89.86) -- (90.07,89.51) -- (88.67,89.86) -- (90.07,90.21) -- (88.67,89.86) -- cycle;
\path[line width=0.24mm, draw=L, fill=F] (98.05,89.86) -- (96.65,90.21) -- (98.05,89.86) -- (96.65,89.51) -- (98.05,89.86) -- cycle;
\path[line width=0.24mm, draw=L] (88.35,57.85) -- (98.27,57.85);
\path[line width=0.24mm, draw=L, fill=F] (88.35,57.85) -- (89.75,57.51) -- (88.35,57.85) -- (89.75,58.20) -- (88.35,57.85) -- cycle;
\path[line width=0.24mm, draw=L, fill=F] (98.27,57.85) -- (96.87,58.20) -- (98.27,57.85) -- (96.87,57.51) -- (98.27,57.85) -- cycle;
\draw(92.23,91.91) node[anchor=base west]{\fontsize{8.54}{10.24}\selectfont $k$};
\draw(92.45,60.55) node[anchor=base west]{\fontsize{8.54}{10.24}\selectfont $p$};
\path[line width=0.15mm, draw=L] (57.64,87.60) -- (45.89,60.66);
\end{tikzpicture}%
\caption{$|[S, \overline{S}]|=t$}
\label{fig1}
\end{figure}

Set 
$$
k = |[U_1, U_2]|,\; p = |[U_3, U_4]|.
$$ 
Then the quotient matrix of the partition $\{U_1, U_2, U_3, U_4\}$ of $V(G)$ is
\begin{equation}
\label{eq:T}
T =
\begin{pmatrix}
r - \frac{k}{n_1}-\frac{t-n_2}{n_1} & \frac{k}{n_1} & \frac{t-n_2}{n_1} & 0 \\
\frac{k}{n_2} & r - \frac{k}{n_2}-1 & 0 & 1 \\
\frac{t-n_2}{n_1} & 0 & r - \frac{p}{n_1}-\frac{t-n_2}{n_1} & \frac{p}{n_1} \\
0 & 1 & \frac{p}{n_2} & r - \frac{p}{n_2}-1
\end{pmatrix}.
\end{equation}
A straightforward computation shows that the characteristic polynomial of $T$ is given by $f(x)=(x-r)P(x)$, where 
\begin{equation}
\begin{split}\label{equ::8}
P(x)=&-n_1^2n_2^2(x-r)^3+[-rn_1n_2(k+p)+2n_1^2n_2^2(n_2-n_1-t)](x-r)^2+\\
&[-kpr^2-n_2(k+p)(n_1^2-n_2^2+t(2n_1+n_2))+4n_1n_2^2(n_2-t)](x-r)+\\
&[2rn_2(n_2-t)(k+p)-2kptr].
\end{split}
\end{equation} 
Let $x_{\max}(k)$ denote the largest real root of $P(x)=0$. Then 
\begin{equation}
\label{eq:GTk}
\lambda_2(G)\geq \lambda_2(T)=x_{\max}(k).
\end{equation}

{\flushleft\bf Claim 3.} $\lambda_2(T)$ is minimized if and only if $k=p=n_1n_2$.

{\flushleft\bf Proof.}
All parameters appeared in the cubic polynomial $P(x)$ are positive integers, and they satisfy
\[
r \le n_1+n_2 \le r+1 \le t,\quad 1\le p\le n_1n_2,\quad 1\le k\le n_1n_2.
\]
We claim that the largest real root $x_{\max}(k)$ of $P(x)=0$ is a strictly decreasing function of $k$.
To prove this, we define 
$$
F(y, z)=G(y)+zH(y),
$$ 
where
\begin{align*}
G(y) &= n_1^2 n_2^2 \, y^3 + \bigl[\, r n_1 n_2 p - 2 n_1^2 n_2^2 (n_2 - n_1 - t) \bigr] y^2 \\
     &\quad + \bigl[\, n_2 p \bigl( n_1^2 - n_2^2 + t(2n_1 + n_2) \bigr) - 4 n_1 n_2^2 (n_2 - t) \bigr] y - 2 r n_2 (n_2 - t) p
\end{align*}
and 
\begin{align*}
H(y) &= r n_1 n_2 \, y^2 + \bigl[\, p r^2 + n_2 \bigl( n_1^2 - n_2^2 + t(2n_1 + n_2) \bigr) \bigr] y + \bigl[\, -2 r n_2 (n_2 - t) + 2 p t r \bigr].
\end{align*}
Then $P(x) = -F(x-r,\,k)$, so the equation $P(x)=0$ becomes $F(x-r,k)=0$. Let $y_0=y_0(z)$ be the largest real root of the cubic equation $F(y,z)=0$ in $y$. Then
\begin{equation}
\label{eq:xmax}
x_{\max}(k)=y_0(k)+r.
\end{equation}
Fix $z>0$ and consider $F(y,z)$ as a cubic polynomial in $y$.
Since its leading coefficient is $n_1^2n_2^2>0$, we have $F(y,z)\to +\infty$ as $y\to +\infty$. If $F_y(y_0,z)=0$, then $y_0$ is a multiple root of $F(y,z)=0$, so $F(y,z)$ does not change sign when crossing $y_0$.
But for $y>y_0$ close to $y_0$, we must have $F(y,z)>0$ because $F(y,z)\to +\infty$ as $y\to +\infty$.
Hence $y_0$ cannot be a multiple (with even multiplicity) rightmost root, and therefore $F_y(y_0,z)\neq 0$.
Moreover, as $F(y,z)$ tends to $+\infty$ as $y\to +\infty$, the rightmost root is crossed from negative to positive, giving $F_y(y_0,z)>0$. By the implicit function theorem, we have
\begin{equation}
\label{eq:dy0}
\frac{d y_0}{d z} = -\frac{F_z(y_0,z)}{F_y(y_0,z)}=-\frac{H(y_0)}{F_y(y_0,z)}.
\end{equation}
Since $F_y(y_0,z)>0$, the sign of $\frac{d y_0}{d z}$ is opposite to the sign of $H(y_0)$.
Thus, to prove that $y_0$ is a strictly decreasing function of $z$, it remains to prove that $H(y_0)>0$.

We now prove that $H(y_0)>0$. The quadratic polynomial $H(y)$ has leading coefficient $rn_1n_2>0$.
Its constant term is
\[
-2r n_2 (n_2 - t) + 2p t r
      =2r\bigl(p t + n_2(t-n_2)\bigr)>0
\]
because $t\ge r+1\ge n_1+n_2\ge n_2$ and $p\ge 1$.
Moreover, the coefficient $pr^2 + n_2(n_1^2-n_2^2+t(2n_1+n_2))$ of the linear term of $H(y)$ is positive. So, if $H(y)=0$ has real roots, they must both be negative, which we denote by $y_1<y_2<0$; if $H(y)=0$ has no real root, then $H(y)>0$ for all $y$ and hence $H(y_0)>0$ as desired.

Consider the case when $H(y)=0$ has real roots $y_1<y_2<0$. Evaluating $F$ at $y_i$ gives $F(y_i,z)=G(y_i)$ as $H(y_i)=0$. Using $t\ge r+1$ and $p\le n_1n_2$,  one can check that
$G(y_i)>0 \; \text{for } i=1,2.$
Hence $F(y_i,z)>0$. Consider the cubic equation $F(y,z)=0$ in $y$.
If it has only one real root, then $\lim_{y\to -\infty}F(y,z)=-\infty$ and $F(y_1,z)>0$
forces the unique real root to lie in $(-\infty,y_1)$, so $y_0<y_1$.
Since $H(y)>0$ for $y\in(-\infty,y_1)$, we get $H(y_0)>0$ as desired.
If the cubic equation $F(y,z)=0$ in $y$ has three real roots, let them be $y_a\le y_b\le y_0$.
Since $\lim_{y\to-\infty}F(y,z)=-\infty$ and $F(y_1,z)>0$, by the intermediate value theorem there exists a root $y_a\in(-\infty,y_1)$. This accounts for one of the three real roots. The remaining two roots $y_b$ and $y_0$ (with $y_b\le y_0$) must both lie in $(y_1,+\infty)$. We now show that in fact $y_b, y_0\in(y_2,+\infty)$, which gives $y_0>y_2$. Since $F(y_1,z)>0$ and $F(y_2,z)>0$, and $F$ is continuous, $F$ does not change sign on $[y_1,y_2]$ (for if it did, there would be a root in $(y_1,y_2)$, and the sign of $F$ would alternate around this root, contradicting $F(y_1,z)>0$ and $F(y_2,z)>0$). Hence $F(y,z)>0$ for all $y\in[y_1,y_2]$, so neither $y_b$ nor $y_0$ lies in $[y_1,y_2]$. Since $y_b,y_0>y_1$ and $y_b,y_0\notin(y_1,y_2]$, we conclude that $y_b,y_0\in(y_2,+\infty)$, and in particular $y_0>y_2$. Since $H(y)>0$ for $y\in(y_2,\infty)$, it follows that $H(y_0)>0$ as desired.

Now that we have proved $H(y_0)>0$, by \eqref{eq:dy0} and $F_y(y_0,z)>0$ we obtain that $\frac{d y_0}{d z}<0$ and hence $y_0(z)$ decreases with $z$. This together with \eqref{eq:xmax} implies that $x_{\text{max}}(k)$ decreases with $k$. Consequently, when $k$ is maximized, $\lambda_2(T) = x_{\text{max}}(k)$ (see \eqref{eq:GTk}) is minimized. Since $k \leq n_1n_2$, it follows that $\lambda_2(T)$ is minimized when $k = n_1n_2$. Similarly, we can prove that $\lambda_2(T)$ is minimized when $p = n_1n_2$. This completes the proof of Claim 3.

{\flushleft\bf Claim 4.} If $k=p=n_1n_2$, then
\begin{equation}
\label{eq:2T}
\lambda_2(T)=\frac{2r n_1 - n_1^2 - n_1 n_2 - 2t - 2n_1 + 2n_2 + \sqrt{N}}{2n_1},
\end{equation}
where 
\begin{equation}
\label{eq:N}
\begin{aligned}
	N = n_1^4 &+ 2n_1^3 n_2 + n_1^2 n_2^2 - 4t n_1^2 + 4t n_1 n_2 + 4n_1^3 \\
	&- 4n_1 n_2^2 + 4t^2 - 8t n_1 - 8t n_2 + 4n_1^2 + 8n_1 n_2 + 4n_2^2.
\end{aligned}
\end{equation}

{\flushleft\bf Proof.}
Suppose that $k=p=n_1n_2$. Then the quotient matrix of the partition $\{U_1, U_2, U_3, U_4\}$ of $V(G)$ (see \eqref{eq:T}) becomes 
\begin{equation}
\label{eq:T1}
T =
\begin{pmatrix}
r - n_2-\frac{t-n_2}{n_1} & n_2 & \frac{t-n_2}{n_1} & 0 \\
n_1 & r - n_1-1 & 0 & 1 \\
\frac{t-n_2}{n_1} & 0 & r - n_2-\frac{t-n_2}{n_1} & n_2 \\
0 & 1 & n_1 & r - n_1-1
\end{pmatrix}.
\end{equation}
Let 
$P=\begin{pmatrix}
I & I \\
I & - I 
\end{pmatrix}$, where $I$ is the $2\times2$ identity matrix. Then
$$P^{-1}TP=\begin{pmatrix}
r - n_2 & n_2 & 0 & 0  \\
n_1 & r - n_1 & 0 & 0  \\
0 & 0  & r - n_2-\frac{2(t-n_2)}{n_1} & n_2 \\
0 & 0  & n_1 & r - n_1-2
\end{pmatrix}=\begin{pmatrix}
A+B & \mathbf{0}  \\
\mathbf{0} & A-B
\end{pmatrix},$$
where 
$$
A=\begin{pmatrix}
r - n_2-\frac{t-n_2}{n_1} & n_2 \\
n_1 & r - n_1-1 
\end{pmatrix},\quad \;
B=\begin{pmatrix}
\frac{t-n_2}{n_1} & 0 \\
0 & 1
\end{pmatrix}
$$ 
and $\mathbf{0}$ is the $2\times2$ zero matrix.
Therefore, the eigenvalues of $T$ are the union of the eigenvalues of $A+B$ and $A-B$. By computing the characteristic polynomials of $A+B$ and $A-B$, we obtain that the eigenvalues of $T$ are $r$, $r-n_1-n_2$ and $\frac{2r n_1 - n_1^2 - n_1 n_2 - 2t - 2n_1 + 2n_2 \pm \sqrt{N}}{2n_1}$, where $N$ is as shown in \eqref{eq:N}. Therefore, $\lambda_2(T)$ is as given in \eqref{eq:2T}. This completes the proof of Claim 4.

{\flushleft\bf Claim 5.} If $t-r$ is odd, then $\lambda_2(T)$ in \eqref{eq:2T} is minimized when $n_1$ is maximized.

{\flushleft\bf Proof.}
Since $t-r$ is odd, we have $n_1+n_2=r+1$. Set $R = r + 1$. Then $t > R = n_1 + n_2$. Substituting $n_2 = R - n_1$ into \eqref{eq:2T} and \eqref{eq:N}, we obtain that
$$
\lambda_2(T) = \frac{(r-5)n_1 + 2r + 2 - 2t + \sqrt{N}}{2n_1}
$$
and
$$
N = (R^2 + 8R - 8t)n_1^2 + 4R(t-R)n_1 + 4(t-R)^2.
$$

Consider the function
$$
g(y)=\frac{f(y)}{2y},
$$
where $f(y)=(r-5)y + 2r + 2 - 2t + \sqrt{N(y)}$ and $N(y)=(R^2 + 8R - 8t)y^2 + 4R(t-R)y + 4(t-R)^2$.
Clearly, $\lambda_2(T)=g(n_1)$. We have
$$
\frac{dg}{dy} = \frac{1}{2y^2} \left(y \frac{df}{dy} - f(y)\right) = \frac{1}{2y^2} \cdot \frac{-B y - 2C - 2D\sqrt{N(y)}}{2\sqrt{N(y)}},
$$
where $B = 4R(t-R) > 0$, $C = 4(t-R)^2 > 0$, and $D = 2R - 2t = -2(t-R) < 0$. Since $[R y + 2(t-R)]^2 - N = 8(t-R) y^2 > 0$, we have $\sqrt{N(y)} < R y + 2(t-R)$, which is equivalent to $-B y - 2C - 2D\sqrt{N(y)} < 0$. Therefore, $\frac{dg}{dy} < 0$. Thus, $g(y)$ is strictly decreasing with $y$. Since $\lambda_2(T)=g(n_1)$, it follows that $\lambda_2(T)$ is strictly decreasing with $n_1$. This proves Claim 5.

{\flushleft\bf Claim 6.} If $t-r$ is even, then $\lambda_2(T)$ in \eqref{eq:2T} is minimized when $n_1$ is maximized.

{\flushleft\bf Proof.}
Since $t-r$ is even, we have $n_1+n_2=r$. Then $t > r = n_1 + n_2$. Substituting $n_2 = r - n_1$ into \eqref{eq:2T} and \eqref{eq:N}, we obtain that
\[
\lambda_2(T) = \frac{(r-5)n_1 + 2r - 2t + \sqrt{N}}{2n_1}
\]
and
\[
N = (r^2 + 8r - 8t)n_1^2 + 4r(t-r)n_1 + 4(t-r)^2.
\]

Consider the function
\[
q(y)=\frac{h(y)}{2y},
\]
where $h(y)=(r-5)y + 2r - 2t + \sqrt{N(y)}$ and $N(y)=(r^2 + 8r - 8t)y^2 + 4r(t-r)y + 4(t-r)^2$.
Clearly, $\lambda_2(T)=g(n_1)$. We have
\[
\frac{dq}{dy} = \frac{1}{2y^2} \left(y \frac{dh}{dy} - h(y)\right) = \frac{1}{2y^2} \cdot \frac{-B y - 2C - 2D\sqrt{N(y)}}{2\sqrt{N(y)}},
\]
where $B = 4r(t-r) > 0$, $C = 4(t-r)^2 > 0$, and $D = 2r - 2t = -2(t-r) < 0$. Since $[r y + 2(t-r)]^2 - N = 8(t-r) y^2 > 0$, we have $\sqrt{N(y)} < r y + 2(t-r)$, which is equivalent to $-B y - 2C - 2D\sqrt{N(y)} < 0$. Therefore, $\frac{dq}{dy} < 0$. Thus, $q(y)$ is strictly decreasing with $y$. Since $\lambda_2(T)=g(n_1)$, it follows that $\lambda_2(T)$ is strictly decreasing with $n_1$. This proves Claim 6.

Finally, we can prove that $\lambda_2(G)\geq \rho(r,t)$. Consider first the case when $t-r$ is odd. In this case we have  $n_1, n_3 \leq t-r-1$. So, by Claim 5, $\lambda_2(T)$ is minimized when $n_1 = n_3 = t-r-1$. Note that when $n_1 = n_3 = t-r-1$ the matrix $T$ in \eqref{eq:T1} becomes
\[
T = 
\begin{pmatrix}
	t-r-4 & 2r-t+2 & 2 & 0 \\
	t-r-1 & 2r-t & 0 & 1 \\
	2 & 0 & t-r-4 & 2r-t+2 \\
	0 & 1 & t-r-1 & 2r-t
\end{pmatrix}.
\]
The second largest eigenvalue of this matrix is $\lambda_2(T) = \frac{r-7+\sqrt{(r+7)^2-8t-32}}{2} =  \rho(r,t)$, where $\rho(r,t)$ is as defined in \eqref{rho}. Thus, by Claims 3 and 5, we have $\lambda_2(G) \geq \rho(r,t)$. 

Now we consider the case when $t-r$ is even. In this case we have $n_1, n_3\leq t-r$. So, by Claim 6, $\lambda_2(T)$ is minimized when $n_1 = n_3 = t-r$. Note that when $n_1 = n_3 = t-r$ the matrix $T$ in \eqref{eq:T1} becomes
\[
T = 
\begin{pmatrix}
t-r-2 & 2r-t & 2 & 0 \\
t-r & 2r-t-1 & 0 & 1 \\
2 & 0 & t-r-2 & 2r-t \\
0 & 1 & t-r & 2r-t-1
\end{pmatrix}.
\]
The second largest eigenvalue of this matrix is $\lambda_2(T) = \frac{r-6+\sqrt{(r+6)^2-8t-32}}{2} = \rho(r,t)$, where $\rho(r,t)$ is as defined in \eqref{rho}. Thus, by Claims 3 and 6, we have $\lambda_2(G) \geq \rho(r,t)$. 

In summary, we have proved that $\lambda_2(G)\geq \rho(r,t)$ for any connected $r$-regular graph $G$ with $\lambda'(G) \leq t$. Since by Lemma \ref{thm::2.2}, $G_{r,t}$ is a connected $r$-regular graph with $\lambda'(G_{r,t}) = t$ and $\lambda_2(G_{r,t}) = \rho(r,t)$, this bound for $\lambda_2(G)$ is sharp. This completes the proof.
\end{proof}

From the above proof of Theorem \ref{thm::1.4}, we see that the bound $\lambda_2(G) \geq \rho(r,t)$ is achieved only when the following conditions hold: (i) if $t-r$ is odd, then $n_1 = n_3 = t-r-1$ and $p = k = (t-r-1)(2r-t+2)$; (ii) if $t-r$ is even, then $n_1 = n_3 = t-r$ and $p = k = (t-r)(2r-t)$. The extremal graphs $G_{r,t}$ defined in Construction \ref{const1} satisfy these conditions and attain the bound $\lambda_2(G) \geq \rho(r,t)$. There are many other connected $r$-regular graphs satisfying conditions (i) and (ii), but we do not believe that all of them have the second largest eigenvalue $\rho(r, t)$. It would be interesting to give, for any integers $6\leq r\leq t \leq2r-3$, a characterization of connected $r$-regular graphs $G$ satisfying $\lambda_2(G) = \rho(r, t)$.

\section*{Acknowledgement}
The first author is supported by China Scholarship Council and Autonomous Region Graduate Research Innovation Project (No. XJ2026G012).


\begin{thebibliography}{99}
	

\bibitem{Abdi02}
M. Abdi, E. Ghorbani, Minimum algebraic connectivity and maximum diameter: Aldous-Fill and Guiduli-Mohar conjectures, \emph{J. Combin. Theory Ser. B} \textbf{167} (2024) 164--188.

\bibitem{Abdi03}
M. Abdi, E. Ghorbani, W. Imrich, Regular graphs with minimum spectral gap, \emph{European J. Combin.} \textbf{95} (2021), 103328.

\bibitem{Abiad01}
A. Abiad, B. Brimkov, X. Martínez-Rivera, O. Suil, J. Zhang, Spectral bounds for the connectivity of regular graphs with given order, \emph{Electron. J. Linear Algebra} \textbf{34} (2018), 428--443.

\bibitem{Aldous02}
D. Aldous, J. Fill, Reversible Markov chains and random walks on graphs, University of California, Berkeley, 2002,
available at \url{https://www.stat.berkeley.edu/~aldous/RWG/book.html}.

\bibitem{AM} 
N. Alon, V. D. Milman. $\lambda_1$, isoperimetric inequalities for graphs, and superconcentrators, \emph{J. Combin. Theory Ser. B} \textbf{38} (1) (1985), 73--88. 

\bibitem{deAbreu01}
N.M.M. de Abreu, Old and new results on algebraic connectivity of graphs, \emph{Linear Algebra Appl.} \textbf{423} (2007), 53--73.


\bibitem{brouwer11}
A.E. Brouwer, W.H. Haemers, \emph{Spectra of Graphs}, Springer, New York, 2011.

\bibitem{Chandran01} 
S.L. Chandran, Minimum cuts, girth and spectral threshold, \emph{Inform. Process. Lett.} \textbf{89} (2004), 105--110.

\bibitem{cioaba10}
S.M. Cioab\u{a}, Eigenvalues and edge-connectivity of regular graphs, \emph{Linear Algebra Appl.} \textbf{432} (2010), 458--470.

\bibitem{J.E} 
J. Ekstein, B. Wu, L.M. Xiong, Connected even factors in the square of essentially $2$-edge-connected graph, \emph{Electron. J. Combin.} \textbf{24} (3) (2017), Paper 3.42, 9 pp.

\bibitem{fiedler01} 
M. Fiedler, Algebraic connectivity of graphs, \emph{Czechoslovak Math. J.} \textbf{23} (1973), 298--305.

\bibitem{X.} 
X.F. Gu, Packing spanning trees and spanning $2$-connected $k$-edge-connected essentially $(2k\!-\!1)$-edge-connected subgraphs, \emph{J. Combin. Optim.} \textbf{33} (2017), 924--933.

\bibitem{X.F.} 
X.F. Gu, R.R. Liu, G.X. Yu, Spanning tree packing and $2$-essential edge-connectivity, \emph{Discrete Math.} \textbf{346} (2023), 113132.

\bibitem{godsil01}
C. Godsil, G. Royle, \emph{Algebraic Graph Theory}, in: Graduate Texts in Mathematics, Vol. 207, Springer-Verlag, New York, 2001.


\bibitem{hoory90}
S. Hoory, N. Linial, A. Wigderson, Expander graphs and their applications, \emph{Bull. Amer. Math. Soc.} \textbf{43} (2006), 439--561.

\bibitem{kirkland}
S. Kirkland, J.J. Molitierno, M. Neumann, B.L. Shader, On graphs with equal algebraic and vertex connectivity, \emph{Linear
Algebra Appl.} \textbf{341} (2002), 45--56.

\bibitem{N.K} 
N. Kothari, M.H. de Carvalho, C.L. Lucchesi, C.H.C. Little. On essentially $4$-edge-connected cubic bricks, \emph{Electron. J. Combin.} \textbf{27} (1) (2020), 22 pp.

\bibitem{krivelevich}
M. Krivelevich, B. Sudakov, Pseudo-random graphs, more sets, in: Graphs and Numbers, \textbf{15} (2006), 199--262.

\bibitem{H-J.} 
H.-J. Lai, J.A. Li, Packing spanning trees in highly essentially connected graphs, \emph{Discrete Math.} \textbf{342} (2019), 1--9.

\bibitem{F.L.} 
F.L. Lu, X. Feng, Y. Wang, $b$-invariant edges in essentially $4$-edge-connected near-bipartite cubic bricks, \emph{Electron. J. Combin.} \textbf{27} (1) (2020), Paper 1.55, 10 pp.

\bibitem{M} 
B. Mohar, Isoperimetric numbers of graphs, \emph{J. Combin. Theory Ser. B} \textbf{47} (3) (1989), 274--291.

\bibitem{mohar01} 
B. Mohar, Some applications of Laplace eigenvalues of graphs, in: G. Hahn, G. Sabidussi (Eds.), \emph{Graph Symmetry:
Algebraic Methods and Applications} \textbf{497} (1997), 225--275.

\bibitem{O10}
S. O, The second largest eigenvalue and vertex-connectivity of regular multigraphs, \emph{Discrete Appl. Math.} \textbf{279} (2020), 118--124.

\bibitem{O90}
S. O, J.R. Park, J. Park, W.Q. Zhang, Sharp spectral bounds for the edge-connectivity of regular graphs, \emph{European J. Combin.} \textbf{110} (2023), 103713.

\bibitem{rad01} 
A.A. Rad, M. Jalili, M. Hasler, A lower bound for algebraic connectivity based on the connection-graph-stability
method, \emph{Linear Algebra Appl.} \textbf{435} (2011), 186--192.

\bibitem{Wu01} 
C.W. Wu, Algebraic connectivity of directed graphs, \emph{Linear Multilinear Algebra} \textbf{53} (2005), 203--223.

\bibitem{J.Q.} 
J.Q. Xu, Z.H. Chen, H.J. Lai, M. Zhang, Spanning trails in essentially $4$-edge-connected graphs, \emph{Discrete Appl. Math.} \textbf{162} (2014), 306--313.

\end{thebibliography}
\end{document}